\documentclass[12pt,reqno]{amsart}

\usepackage{amsmath, amssymb, amsthm,mathtools}
\usepackage{enumitem}
\usepackage{fullpage}
\usepackage[numbers,sort&compress]{natbib}
\usepackage[colorlinks=true,linkcolor=blue,citecolor=blue,urlcolor=blue]{hyperref}
\usepackage{orcidlink}

\newtheorem{theorem}{Theorem}[section]
\newtheorem{lemma}[theorem]{Lemma}
\newtheorem{proposition}[theorem]{Proposition}
\newtheorem{corollary}[theorem]{Corollary}
\theoremstyle{definition}

\newtheorem{remark}[theorem]{Remark}
\newtheorem{question}{Question}
\newtheorem{example}[theorem]{Example}

\begin{document}

\title{Asymptotic Numerical Ranges and Invariant Subspaces of Operators}
\thanks{Supported by the National Natural Science Foundation of China (No.~12271202).}

\author[Y. Ji]{Youqing Ji}
\address{School of Mathematics\\
Jilin University\\
Changchun 130012\\
P.R. China}
\email{jiyq@jlu.edu.cn}

\author[B. Yang]{Bangyuan Yang$^{*}$\,\orcidlink{0009-0002-0355-4458}}
\address{School of Mathematics\\
Jilin University\\
Changchun 130012\\
P.R. China}
\email{yangby23@mails.jlu.edu.cn}
\email{byangmath@gmail.com}
\date{}
\thanks{$^{*}$Corresponding author.}


\begin{abstract}
For a bounded linear operator $T$ on a complex separable Hilbert space $\mathcal{H}$ and a vector $x\in \mathcal{H}$, let $W_a(T,x)$ be the set of cluster points of the sequence $\{\langle|T^n|^{1/n}x,x\rangle\}_{n=1}^{\infty}$. We define the asymptotic numerical range and the asymptotic numerical radius of $T$, respectively, by $W_a(T):=\bigcup_{\|x\|=1}W_a(T,x)$ and $w_a(T):=\sup W_a(T)$. We prove that $W_a(T,x)$ is a compact interval for every $x\in\mathcal{H}$ and that $W_a(T)$ is a bounded interval, where intervals are allowed to be degenerate. Using the asymptotic numerical radius, we show that if there is a nonzero vector $x_0\in \mathcal{H}$ with $r(T,x_0)<w_a(T)$, where $r(T,x_0)$ denotes the local spectral radius of $T$ at $x_0$, then the subspaces $\overline{\{x\in\mathcal{H}:r(T,x)\le r(T,x_0)\}}$ and $\overline{\operatorname{span}\{T^n x_0:n\ge0\}}$, which are well known to be hyperinvariant and invariant for $T$, respectively, are both nontrivial. We also prove that $w_a(T)=r(T)$ for every hyponormal operator $T$, where $r(T)$ denotes the spectral radius of $T$. Finally, we investigate the connectedness of the set of WOT cluster points of the sequence $\{|T^n|^{1/n}\}_{n=1}^\infty$.
\end{abstract}

\maketitle

\section{Introduction}\label{intro}

Throughout this paper, we denote by $\mathcal{H}$ a complex separable Hilbert space endowed with the inner product $\langle\cdot,\cdot\rangle$, and by $\mathcal{B}(\mathcal{H})$ the algebra of all bounded linear operators on $\mathcal{H}$. The identity operator is denoted by $I$. For $A\in\mathcal{B}(\mathcal{H})$, we write $|A|\coloneqq(A^*A)^{1/2}$, and $\sigma(A)$ and $r(A)$ denote the spectrum and the spectral radius of $A$, respectively. We abbreviate the weak and strong operator topologies on $\mathcal{B}(\mathcal{H})$ by WOT and SOT, respectively. Recall that $A$ is said to be \emph{hyponormal} if $A^*A\ge AA^*$. The \emph{numerical range} of $A$ is defined by $W(A)\coloneqq\{\langle Ax,x\rangle:x\in\mathcal{H},\ \|x\|=1\}$, and the \emph{numerical radius} of $A$ is $w(A)\coloneqq\sup\{|\lambda|:\lambda\in W(A)\}$. For $x\in\mathcal{H}$, the \emph{local spectral radius} of $A$ at $x$ is defined by $r(A,x)\coloneqq\limsup_{n\to\infty}\|A^nx\|^{1/n}$.
By an \emph{interval} we mean a nonempty convex (equivalently, connected) subset of $\mathbb{R}$; in particular, singletons are also regarded as (degenerate) intervals.

A \emph{linear manifold} in $\mathcal{H}$ is a linear subset of $\mathcal{H}$ that need not be closed in norm; a norm-closed linear manifold is called a \emph{subspace}. Given $T\in\mathcal{B}(\mathcal{H})$, a linear manifold $M$ is said to be \emph{invariant} for $T$ if $T(M)\subseteq M$, and $M$ is said to be \emph{hyperinvariant} for $T$ if $S(M)\subseteq M$ for every $S\in\mathcal{B}(\mathcal{H})$ commuting with $T$. A subspace of $\mathcal{H}$ is said to be \emph{nontrivial} if it is different from both $\{0\}$ and $\mathcal{H}$.

The invariant subspace problem, which remains open, asks whether every operator in $\mathcal{B}(\mathcal{H})$ has a nontrivial invariant subspace. One approach to this problem is to associate suitable characteristic sets or asymptotic quantities with an operator and to exploit their structural properties in the construction of invariant or hyperinvariant subspaces.

Let $T\in \mathcal{B}(\mathcal{H})$. A classical consequence of the Riesz decomposition theorem, originating from the functional calculus introduced by Riesz in 1913, is that if the spectrum of $T$ is disconnected, then the associated Riesz projection yields a nontrivial hyperinvariant subspace; see \cite[Proposition~VII.4.11(a)]{conway1990Course}. Let $\mathbb D$ denote the open unit disk of the complex plane and let $H^\infty$ denote the algebra of bounded analytic functions on $\mathbb D$, equipped with the supremum norm $\|h\|_\infty\coloneqq\sup_{\lambda\in\mathbb D}|h(\lambda)|$. Brown, Chevreau, and Pearcy \cite[Theorem~4.1]{brown1979Contractions} proved that if $T$ satisfies $\|T\|=1$ and if $\sigma(T)\cap\mathbb D$ is sufficiently large in the sense that
\[
\sup_{\lambda\in\sigma(T)\cap\mathbb D}|h(\lambda)|=\|h\|_\infty
\]
for every $h\in H^\infty$, then $T$ has a nontrivial invariant subspace. Brown \cite[Theorem~4]{brown1987Hyponormal} subsequently proved that if $T$ is hyponormal and $R(\sigma(T))\neq C(\sigma(T))$, where $C(\sigma(T))$ denotes the space of continuous functions on $\sigma(T)$ and $R(\sigma(T))$ denotes the uniform closure in $C(\sigma(T))$ of the rational functions with poles off $\sigma(T)$, then $T$ has a nontrivial invariant subspace. In particular, this applies whenever $T$ is hyponormal and $\sigma(T)$ has nonempty interior. Brown \cite{brown1988Contractions} also obtained an asymptotic criterion: if $T$ is a $C_{00}$-contraction, that is, a contraction such that both $T^n\to0$ and $T^{*n}\to0$ in the SOT, and if $\sigma(T)$ contains the unit circle, then $T$ has a nontrivial invariant subspace. Eschmeier and Prunaru \cite[Corollary~3.4]{eschmeier2002Invariant} proved that if the localizable spectrum of $T$ is thick in the sense of Brown's theory, then $T$ has a nontrivial invariant subspace. In the setting of a type $\mathrm{II}_1$ factor $\mathfrak M$, Haagerup and Schultz \cite[Corollary~7.2]{haagerup2009Invariant} proved that if $T\in\mathfrak M$ and its Brown measure is not concentrated on a singleton, then $T$ admits a nontrivial hyperinvariant subspace. Douglas and Yang \cite{douglas2021Hermitian} introduced, for $T$ with $0$ as an isolated spectral point and $x\in\mathcal H\setminus\{0\}$,
\[
k_x\coloneqq\limsup_{z\to0}\frac{\log\|(T-z)^{-1}x\|^2}{\log\|(T-z)^{-1}\|^2},\qquad\Lambda(T)\coloneqq\{k_x:x\in\mathcal H,\, x\ne0\},
\]
where $\Lambda(T)$ is called the power set of $T$, and proved that, for $0\le\tau\le1$, the subsets
\[
N_\tau\coloneqq\{x\in\mathcal H:x=0\ \text{or}\ k_x\le\tau\}
\]
are hyperinvariant linear manifolds. In particular, if $T$ is quasinilpotent and $\Lambda(T)$ contains two distinct points $\tau_1<\tau_2$ such that both $N_{\tau_1}$ and $N_{\tau_2}$ are closed, then $T$ has a nontrivial hyperinvariant subspace; see \cite[Proposition~7.1 and Corollary~7.2]{douglas2021Hermitian}. The relationship between power sets and invariant subspaces was also studied in \cite{liang2018Quasinilpotent,ji2021Power,he2022Power,hu2024Power}.

Bourdon \cite[Proposition~4.6]{bourdon1997Orbits} proved that if $T$ is hyponormal and
\begin{equation}\label{bourdoncri}
r(T,x_0)<\|T\| \qquad\text{for some }x_0\ne0,
\end{equation}
then $T$ has a nontrivial hyperinvariant subspace.

Motivated by these results, it is natural to ask the following.

\begin{question}\label{whichtau}
Which further characteristic sets or asymptotic quantities attached to an operator can be used to produce nontrivial invariant or hyperinvariant subspaces?
\end{question}

We propose one such quantity, based on the asymptotic behavior of the sequence $\{|T^n|^{1/n}\}_{n=1}^{\infty}$. For a general operator $T\in \mathcal{B}(\mathcal{H})$ and $x\in \mathcal{H}$, we define
\[
W_a(T,x)\coloneqq\operatorname{Clust}\left\{\left\langle |T^n|^{1/n}x,x\right\rangle\right\}_{n=1}^{\infty}.
\]
Here, for a sequence $\{t_n\}_{n=1}^{\infty}$ in a metric space $(X,d)$, $\operatorname{Clust}\{t_n\}_{n=1}^{\infty}$ denotes the set of all its subsequential limits. Equivalently,
\[
\operatorname{Clust}\{t_n\}_{n=1}^{\infty}=\bigcap_{N=1}^{\infty}\overline{\{t_n:n\ge N\}}.
\]
Thus this set is always closed, and if $\{t_n\}_{n=1}^{\infty}$ is contained in a compact subset of $X$, then $\operatorname{Clust}\{t_n\}_{n=1}^{\infty}$ is nonempty compact and reduces to a singleton precisely when $\{t_n\}_{n=1}^{\infty}$ converges. We then set
\[
W_a(T)\coloneqq\bigcup_{\|x\|=1}W_a(T,x), \qquad w_a(T)\coloneqq\sup W_a(T),
\]
and call them the \emph{asymptotic numerical range} and the \emph{asymptotic numerical radius} of $T$, respectively.

For $\tau\ge0$, set
\[
M_\tau(T)\coloneqq\{x\in\mathcal{H}:r(T,x)\le\tau\}.
\]
We will prove that if
\begin{equation}\label{ourcri}
r(T,x_0)<w_a(T) \qquad\text{for some }x_0\ne0,
\end{equation}
then $T$ has the nontrivial hyperinvariant subspace $\overline{M_{r(T,x_0)}(T)}$ and the nontrivial invariant subspace $\overline{\operatorname{span}\{T^n x_0:n\ge0\}}$; see Theorem~\ref{nontriinv}. This is an answer to Question~\ref{whichtau}. Like the criteria introduced above by Douglas and Yang and by Bourdon, our criterion \eqref{ourcri} relies on comparing the sizes of suitable quantities. We prove the general comparison $w_a(T)\le r(T)$; see Proposition~\ref{impoineq}.
We also show that if $\{|T^n|^{1/n}\}_{n=1}^{\infty}$ has a subsequence $\{|T^{n_j}|^{1/n_j}\}_{j=1}^{\infty}$ which is an increasing operator sequence, then $w_a(T)=r(T)$; if, in addition, $n_1=1$, then
\[
w_a(T)=r(T)=w(T)=\|T\|;
\]
see Proposition~\ref{hyponequi}. As a consequence, the latter equalities hold for every hyponormal operator. In fact, they remain valid for a broader class of operators; see Corollary~\ref{phyponormal}. Thus, for hyponormal operators, our criterion \eqref{ourcri} coincides with Bourdon's criterion \eqref{bourdoncri}, so our result recovers Bourdon's theorem through a different approach and extends its conclusion.

In Example~\ref{mpv}, we apply \eqref{ourcri} to the sum $T$ of the multiplication operator by the independent variable and the Volterra operator on $L^2([0,1])$ and obtain a class of nontrivial hyperinvariant subspaces of $T$, which forms a subclass of the invariant subspace lattice characterized by Sarason \cite[Section~9]{sarason1974Invariant}.

Beyond its application to invariant subspaces, it is natural to ask what intrinsic structure the asymptotic numerical range possesses. Indeed, other sets arising in operator theory have themselves been the subject of structure theorems: the spectrum is always compact, the numerical range is convex by the Toeplitz--Hausdorff theorem \cite{toeplitz1918Algebraische,hausdorff1919Wertvorrat}, and Ji and Zhang \cite{ji2023Power} recently proved that the power set of a quasinilpotent operator is a right-closed subset of $[0,1]$, and conversely that every right-closed subset of $[0,1]$ containing $1$ is the power set of some quasinilpotent operator. In local spectral theory, Dane\v{s} \cite[Proposition~2]{danes1987Local} proved that, for each fixed vector $x$, the set of cluster points of the sequence $\{\|T^nx\|^{1/n}\}_{n=1}^{\infty}$ is a compact interval.

\begin{question}\label{strucque}
What geometric and topological properties does the asymptotic numerical range possess?
\end{question}

We will show that $W_a(T,x)$ is a compact interval for every $x\in \mathcal{H}$ and that $W_a(T)$ is a bounded interval (see Theorems~\ref{Qxcompactinterval} and~\ref{asyint}), an asymptotic analogue of the Toeplitz--Hausdorff theorem; this partially answers Question~\ref{strucque}.

The asymptotic behavior of the sequence $\{|T^n|^{1/n}\}_{n=1}^{\infty}$ has also been investigated in connection with invariant subspaces. Haagerup and Schultz \cite[Theorem~8.1]{haagerup2009Invariant} proved that, for an operator $T$ in a type $\mathrm{II}_1$ factor, the sequence $\{|T^n|^{1/n}\}_{n=1}^{\infty}$ converges in the SOT to a positive operator $A$. Moreover, the spectral projections of $A$ are closely related to the hyperinvariant subspaces constructed in their work.

A different line of research concerns the intrinsic asymptotic properties of the sequence $\{|T^n|^{1/n}\}_{n=1}^{\infty}$. In finite dimensions, Yamamoto \cite{yamamoto1967Extreme} obtained the following generalization of the spectral radius formula: for $T\in \mathcal{B}(\mathbb{C}^m)$,
\[
\lim_{n\to\infty}s_j(T^n)^{1/n}=|\lambda_j(T)|\qquad(1\le j\le m),
\]
where $s_j(\cdot)$ denotes the $j$-th largest singular value and the eigenvalues $\lambda_j(T)$ are arranged in decreasing order of modulus (counted with
multiplicity). The corresponding statement for compact operators on $\mathcal{H}$ is due to Davis \cite{davis1970Theorem}. These are results at the level of the individual singular values of the sequence $\{|T^n|^{1/n}\}_{n=1}^{\infty}$, arranged in decreasing order. In recent years, it has been shown that $\{|T^n|^{1/n}\}_{n=1}^{\infty}$ is convergent in norm in several settings: Nayak \cite{nayak2023Stronger} established this for $T\in\mathcal B(\mathbb C^m)$, Bala and Bhat \cite{bala2025Survey} subsequently for every compact operator $T$ on $\mathcal{H}$, and, more recently, Nayak and Shekhawat~\cite{nayak2026Convergence} for every spectral operator $T$ on $\mathcal{H}$.

For a general $T\in \mathcal{B}(\mathcal{H})$, however, no convergence of this kind is available. We consider the WOT cluster set
\[
\mathfrak C\coloneqq\operatorname{Clust}_{\operatorname{WOT}}\{|T^n|^{1/n}\}_{n=1}^{\infty}.
\]
There is no ambiguity in this notation, and $\mathfrak C$ reduces to a singleton precisely when $\{|T^n|^{1/n}\}_{n=1}^{\infty}$ converges in the WOT; see Remark~\ref{welldef}. It serves as the main intermediate object connecting the two questions above: we show that
\[
W_a(T,x)=\{\langle Ax,x\rangle:A\in\mathfrak C\} \quad\text{for every }x\in \mathcal{H}, \qquad\text{and}\qquad W_a(T)=\bigcup_{A\in\mathfrak C}W(A);
\]
see Proposition~\ref{WaC}.

\begin{question}\label{clusque}
What is the structure of the cluster set $\mathfrak C$, and what is $\mathfrak C$ for concrete operators?
\end{question}

Since $W_a(T,x)$ is a compact interval for every $x$ and $W_a(T)$ is a bounded interval, it is natural to ask whether $\mathfrak C$ is connected in the WOT. The set $\mathfrak C$ is indeed WOT-connected when $T$ is bounded below (see Corollary~\ref{bblcon} and Proposition~\ref{boubelcon}), whereas in general connectedness may fail (see Example~\ref{wsnbb}). For bounded below unilateral weighted shifts we describe $\mathfrak C$ completely; see Proposition~\ref{weishi}.

\section{Asymptotic Numerical Range}\label{asynur}

Throughout, $T\in\mathcal{B}(\mathcal{H})$ denotes a fixed operator unless otherwise specified.

For a real number $R\ge0$, write $\overline{\operatorname{ball}}_R\coloneqq\{A\in \mathcal{B}(\mathcal{H}):\|A\|\le R\}$. Recall that, since $\mathcal H$ is separable, the WOT restricted to each of these balls is metrizable; more precisely, we have the following standard result.

\begin{lemma}[{\cite[I.6]{davidson1996CAlgebras}}]\label{wotball}
Fix a dense sequence $\{x_i\}_{i=1}^{\infty}$ in the closed unit ball of $\mathcal{H}$, and define
\begin{equation}\label{wotmetric}
d(A,B)\coloneqq\sum_{i,j\ge1}2^{-i-j}\,\bigl|\langle (A-B)x_i,x_j\rangle\bigr| \qquad(A,B\in \mathcal{B}(\mathcal{H})).
\end{equation}
Then $d$ is a metric on $\mathcal{B}(\mathcal{H})$. For every $R\ge 0$, the space $(\overline{\operatorname{ball}}_R,\operatorname{WOT})$ is compact (and therefore closed in $(\mathcal{B}(\mathcal{H}),\operatorname{WOT})$, which is Hausdorff). Moreover, the restriction of $d$ to $\overline{\operatorname{ball}}_R$ metrizes the $\operatorname{WOT}$ on $\overline{\operatorname{ball}}_R$. Thus, $(\overline{\operatorname{ball}}_R,\operatorname{WOT})$ is sequentially compact.
\end{lemma}

\begin{remark}\label{welldef}
There is no ambiguity in the notation
\[
\operatorname{Clust}_{\operatorname{WOT}}\{A_n\}_{n=1}^{\infty}
\]
for a norm-bounded sequence $\{A_n\}_{n=1}^{\infty}$ in $\mathcal{B}(\mathcal{H})$. More precisely, choose $R\ge0$ such that $\{A_n\}_{n=1}^{\infty}\subseteq\overline{\operatorname{ball}}_R$. The WOT cluster set of $\{A_n\}_{n=1}^{\infty}$ may be computed inside $\overline{\operatorname{ball}}_R$, inside any other $\overline{\operatorname{ball}}_{R'}$ containing the sequence, or inside $\mathcal{B}(\mathcal{H})$ itself; all three choices give the same set. Indeed, by Lemma~\ref{wotball}, it suffices to use the following elementary topological facts: if $Y$ is a compact subset of a Hausdorff space $X$ and $\{t_n\}_{n=1}^{\infty}\subseteq Y$, then every subnet limit in $X$ of $\{t_n\}_{n=1}^{\infty}$ already lies in $Y$; and if $Y$ is moreover metrizable, every such limit is in fact the limit of a subsequence of $\{t_n\}_{n=1}^{\infty}$. Applying these facts with $X=(\mathcal{B}(\mathcal{H}),\operatorname{WOT})$ and $Y=(\overline{\operatorname{ball}}_R,\operatorname{WOT})$ proves the assertion. In particular, $\operatorname{Clust}_{\operatorname{WOT}}\{A_n\}_{n=1}^{\infty}$ is a nonempty WOT closed subset of $\overline{\operatorname{ball}}_R$, and it reduces to a singleton precisely when $\{A_n\}_{n=1}^{\infty}$ converges in the WOT.
\end{remark}

Since $|T^n|^{1/n}\in\overline{\operatorname{ball}}_{\|T\|}$ for all $n$, Remark~\ref{welldef} applies with $A_n=|T^n|^{1/n}$.

\begin{proposition}\label{WaC}
For every $x\in \mathcal{H}$,
\[
W_a(T,x)=\{\langle Ax,x\rangle:A\in\mathfrak C\}.
\]
Moreover,
\[
W_a(T)=\bigcup_{A\in\mathfrak C}W(A).
\]
\end{proposition}

\begin{proof}
Set $A_n\coloneqq|T^n|^{1/n}$. If $A\in\mathfrak C$, then by Remark~\ref{welldef} there exists a subsequence $\{n_j\}$ such that $A_{n_j}\xrightarrow{\operatorname{WOT}}A$. Hence $\langle A_{n_j}x,x\rangle\to\langle Ax,x\rangle$, so $\langle Ax,x\rangle\in W_a(T,x)$.

Conversely, let $\lambda\in W_a(T,x)$. Then $\langle A_{n_j}x,x\rangle\to\lambda$ for some subsequence $\{n_j\}$. Since $\{A_n\}\subseteq\overline{\operatorname{ball}}_{\|T\|}$, Lemma~\ref{wotball} allows us, after passing to a further subsequence, to assume that $A_{n_j}\xrightarrow{\operatorname{WOT}}A$ for some $A\in\mathfrak C$. Therefore, $\lambda=\langle Ax,x\rangle$, proving the first identity.

Taking the union over all unit vectors gives
\[
W_a(T)=\bigcup_{\|x\|=1}W_a(T,x)=\bigcup_{\|x\|=1}\bigcup_{A\in\mathfrak{C}}\left\{\langle A x,x\rangle\right\}=\bigcup_{A\in\mathfrak C}W(A).
\]
\end{proof}

The main purpose of this section is to establish the convexity of the asymptotic numerical range. We begin with its pointwise version.
\begin{theorem}\label{Qxcompactinterval}
For every $x\in \mathcal{H}$, $W_a(T,x)$ is a compact interval. More precisely,
\[
W_a(T,x)
=
\left[
\liminf_{n\to\infty}\left\langle |T^n|^{1/n}x,x\right\rangle,
\,\limsup_{n\to\infty}\left\langle |T^n|^{1/n}x,x\right\rangle
\right].
\]
\end{theorem}

In order to prove Theorem~\ref{Qxcompactinterval}, we recall the following classical criterion.

\begin{lemma}[Barone {\cite[Theorem~3.2]{barone1939Limit}}]\label{barone}
Let $\{t_n\}_{n=1}^{\infty}$ be a sequence in a metric space $(X,d)$ such that $\{t_n:n\ge1\}$ is contained in a compact subset of $X$. If
\[
\lim_{n\to\infty}d(t_{n+1},t_n)=0,
\]
then $\operatorname{Clust}\{t_n\}_{n=1}^{\infty}$ is connected.
\end{lemma}

In our application to bounded sequences of real numbers, it is enough to control only the ``upward jumps''. We shall use the following one-sided variant.

\begin{lemma}\label{onesidclupoi}
Let $\{s_n\}_{n=1}^{\infty}$ be a bounded sequence of real numbers. If
\[
\limsup_{n\to\infty}(s_{n+1}-s_n)\le 0,
\]
then $\operatorname{Clust}\{s_n\}_{n=1}^{\infty}$ is a compact interval; that is,
\[
\operatorname{Clust}\{s_n\}_{n=1}^{\infty}
=
\bigl[\liminf_{n\to\infty}s_n,\,
\limsup_{n\to\infty}s_n\bigr].
\]
\end{lemma}

\begin{proof}
Let $E\coloneqq\operatorname{Clust}\{s_n\}_{n=1}^{\infty}$. Since $\{s_n\}_{n=1}^{\infty}$ is bounded, $E$ is a nonempty compact subset of $\mathbb R$. It remains to prove that $E$ is connected.

Assume, to the contrary, that $E$ is not connected. Since $E$ is a compact subset of $\mathbb R$, there exist $a<b$ in $E$ such that
\[
[a,b]\cap E=\{a,b\}.
\]
Choose numbers $a_1,b_1$ such that
\[
a<a_1<b_1<b.
\]
Define
\[
L\coloneqq\{n:s_n<a_1\}, \qquad R\coloneqq\{n:s_n>b_1\}.
\]
It is easy to see that $L$ and $R$ are disjoint infinite sets. Moreover, there are at most finitely many indices outside $L\cup R$.

Therefore there exists a strictly increasing sequence of positive integers $\{n_k\}_{k=1}^{\infty}$ such that
\[
n_k\in L \qquad\text{and}\qquad n_k+1\in R
\]
for every $k$. Consequently,
\[
s_{n_k+1}-s_{n_k}>b_1-a_1>0
\]
for every $k$. This contradicts the assumption
\[
\limsup_{n\to\infty}(s_{n+1}-s_n)\le0.
\]
This contradiction shows that $E$ is connected and therefore a compact interval.
\end{proof}

We shall also need the following two standard operator inequalities. For proofs, see \cite[Corollary~I.4.6]{davidson1996CAlgebras} and \cite{pedersen1972Shorter}, respectively.

\begin{lemma}\label{opineq}
\begin{enumerate}[label=\textup{(\roman*)},leftmargin=*]
\item If $A\le B$, then, for every $X\in \mathcal{B}(\mathcal{H})$,
\begin{equation}\label{cong}
X^*AX\le X^*BX.
\end{equation}
\item \textup{(L\"owner--Heinz inequality)} The function $t\mapsto t^{\alpha}$ is operator monotone on $[0,\infty)$ for every $0<\alpha\le1$; that is, if $A\ge B\ge0$ and $0<\alpha\le1$, then
\begin{equation}\label{lhe}
A^{\alpha}\ge B^{\alpha}.
\end{equation}
\end{enumerate}
\end{lemma}

We will use Lemma~\ref{onesidclupoi} to prove Theorem~\ref{Qxcompactinterval}; the next two lemmas provide the required estimates.

\begin{lemma}\label{nexpowgen}
For every $x\in \mathcal{H}$,
\[
\limsup_{n\to\infty} \left\langle \left( |T^{n+1}|^{1/(n+1)}-|T^n|^{1/(n+1)} \right)x,x \right\rangle \le 0.
\]
\end{lemma}

\begin{proof}
Set $M\coloneqq\max\{\|T\|,1\}$. Then $T^*T\le M^2I$, and hence, by inequality~\eqref{cong},
\[
|T^{n+1}|^2 = (T^n)^*T^*T(T^n) \le M^2(T^n)^*T^n=M^2|T^n|^2.
\]
Applying the L\"owner--Heinz inequality \eqref{lhe} with exponent $\alpha=1/(2n+2)$, we get
\[
|T^{n+1}|^{1/(n+1)} \le M^{1/(n+1)}|T^n|^{1/(n+1)}.
\]
Therefore
\[
|T^{n+1}|^{1/(n+1)}-|T^n|^{1/(n+1)}
\le \left(M^{1/(n+1)}-1\right)|T^n|^{1/(n+1)}.
\]
By the choice of $M$, we have $|T^n|^{1/(n+1)}\le M^{n/(n+1)}I\le MI$, and therefore
\[
\left(M^{1/(n+1)}-1\right)|T^n|^{1/(n+1)}
\le M\left(M^{1/(n+1)}-1\right)I.
\]

Thus, for every $x\in \mathcal{H}$,
\[
\left\langle \left( |T^{n+1}|^{1/(n+1)}-|T^n|^{1/(n+1)} \right)x,x \right\rangle \le M\left(M^{1/(n+1)}-1\right)\|x\|^2.
\]
It follows that
\[
\limsup_{n\to\infty} \left\langle \left( |T^{n+1}|^{1/(n+1)}-|T^n|^{1/(n+1)} \right)x,x \right\rangle \le0.
\]
\end{proof}

\begin{lemma}\label{convparcomoper}
\[
\lim_{n\to\infty}\bigl\||T^n|^{1/n}-|T^n|^{1/(n+1)}\bigr\|=0.
\]
\end{lemma}

\begin{proof}
Set
\[
A_n\coloneqq|T^n|^{1/n}, \qquad M\coloneqq\max\{\|T\|,1\}.
\]
Then $A_n$ is positive and
\[
\sigma(A_n)\subseteq [0,\|A_n\|]\subseteq [0,\|T\|]\subseteq [0,M].
\]
By the continuous functional calculus, we obtain
\[
\bigl\||T^n|^{1/n}-|T^n|^{1/(n+1)}\bigr\|= \bigl\|A_n-A_n^{n/(n+1)}\bigr\|=\sup_{r\in\sigma(A_n)} \left|r-r^{n/(n+1)}\right|\le\sup_{0\le r\le M} \left|r-r^{n/(n+1)}\right|.
\]

We now estimate the last supremum. Define $f_n(r)\coloneqq\bigl|r-r^{n/(n+1)}\bigr|$ for $r\in [0,M]$. The continuous function $f_n$ is differentiable on $(0,1)\cup(1,M)$, and
\[
f_n'(r) =
\begin{cases}
\dfrac{n}{n+1}r^{-1/(n+1)}-1, & 0<r<1,\\[8pt]
1-\dfrac{n}{n+1}r^{-1/(n+1)}, & 1<r<M.
\end{cases}
\]
Thus the only critical point of $f_n$ is
\[
r_n\coloneqq\left(\frac{n}{n+1}\right)^{n+1}\in(0,1).
\]
Since $f_n$ attains its maximum on $[0,M]$, every maximum point must be an endpoint, the critical point $r_n$, or the point $1$, where $f_n$ may fail to be differentiable. Hence
\[
\sup_{0\le r\le M} f_n(r) = \max\{f_n(0),\,f_n(r_n),\,f_n(1),\,f_n(M)\}.
\]
Now
\[
f_n(0)=f_n(1)=0,
\]
and
\[
f_n(r_n) = r_n^{n/(n+1)}-r_n = \left(\frac{n}{n+1}\right)^n - \left(\frac{n}{n+1}\right)^{n+1} = \frac{1}{n+1} \left(\frac{n}{n+1}\right)^n.
\]
Moreover, since $M\ge 1$,
\[
f_n(M)=M-M^{n/(n+1)}.
\]
Therefore,
\[
\lim_{n\to\infty}\left(\sup_{0\le r\le M} f_n(r)\right)
=
\lim_{n\to\infty}
\max\left\{
\frac{1}{n+1}\left(\frac{n}{n+1}\right)^n,
\,M-M^{n/(n+1)}
\right\}
=0.
\]

Consequently,
\[
\lim_{n\to\infty}\bigl\||T^n|^{1/n}-|T^n|^{1/(n+1)}\bigr\|=0.
\]
\end{proof}

\begin{proof}[\textbf{\textup{Proof of Theorem~\ref{Qxcompactinterval}}}]
It only remains to apply Lemma~\ref{onesidclupoi}. Fix $x\in\mathcal{H}$ and set
\[
s_n\coloneqq\left\langle |T^n|^{1/n}x,x\right\rangle.
\]
The boundedness of $\{s_n\}_{n=1}^{\infty}$ follows from
\[
0\le s_n \le \||T^n|^{1/n}\|\,\|x\|^2 \le \|T\|\,\|x\|^2.
\]
We write
\[
s_{n+1}-s_n = \left\langle \left(|T^{n+1}|^{1/(n+1)}-|T^n|^{1/(n+1)}\right)x,x \right\rangle + \left\langle \left(|T^n|^{1/(n+1)}-|T^n|^{1/n}\right)x,x \right\rangle.
\]
By Lemmas~\ref{nexpowgen} and~\ref{convparcomoper},
\[
\limsup_{n\to\infty}(s_{n+1}-s_n)\le0.
\]
The conclusion now follows from Lemma~\ref{onesidclupoi} and the definition of $W_a(T,x)$.
\end{proof}

Theorem~\ref{Qxcompactinterval} yields the following asymptotic analogue of the Toeplitz--Hausdorff theorem.

\begin{theorem}\label{asyint}
The asymptotic numerical range $W_a(T)$ is a bounded interval.
\end{theorem}
\begin{proof}
Since $W_a(T)\subseteq[0,\|T\|]$, it suffices to show that $W_a(T)$ is connected, that is, $W_a(T)$ has only one connected component.

By Proposition~\ref{WaC}, write two arbitrary points of $W_a(T)$ as $\langle A_1x,x\rangle$ and $\langle A_2y,y\rangle$, where $A_1,A_2\in\mathfrak C$ and $\|x\|=\|y\|=1$. By Theorem~\ref{Qxcompactinterval}, $W_a(T, x)\subseteq W_a(T)$ is a compact interval containing both $\langle A_1 x, x\rangle$ and $\langle A_2 x, x\rangle$. Since $W(A_2)\subseteq W_a(T)$ is an interval (by the Toeplitz--Hausdorff theorem) containing both $\langle A_2 x, x\rangle$ and $\langle A_2 y, y\rangle$, all three points lie in the same connected component of $W_a(T)$. Hence $W_a(T)$ has a single connected component, and is therefore a bounded interval.
\end{proof}

\begin{remark}
If $T$ is quasinormal, i.e., $T^*TT=TT^*T$, then a simple induction gives $|T^n|^{1/n}\equiv|T|$ for all $n$, whence $W_a(T)=W(|T|)$. In particular, if $J$ is of one of the following forms: $(a,b)$, $[a,b)$, $(a,b]$, or $[a,b]$, with $0\le a<b<\infty$, or $\{a\}$, with $0\le a<\infty$, then one can easily construct a positive diagonal operator $D$ such that $W(D)=J$. Consequently, every bounded interval in $[0,\infty)$ can be realized as the asymptotic numerical range of some operator.
\end{remark}

The proof of Lemma~\ref{nexpowgen} suggests the following stronger conclusion under the additional assumption that $T$ is bounded below.

\begin{lemma}\label{nexpowbb}
If $T$ is bounded below, then
\[
\lim_{n\to\infty}\bigl\| |T^{n+1}|^{1/(n+1)}-|T^n|^{1/(n+1)} \bigr\|=0.
\]
\end{lemma}

\begin{proof}
Set $M\coloneqq\max\{\|T\|,1\}$. The estimate established in the proof of Lemma~\ref{nexpowgen} gives
\begin{equation}\label{upperbd}
|T^{n+1}|^{1/(n+1)}-|T^n|^{1/(n+1)}
\le M\left(M^{1/(n+1)}-1\right)I.
\end{equation}
Since $T$ is bounded below, set
\[
c\coloneqq\inf_{\|x\|=1}\|Tx\|>0,
\]
and set $m\coloneqq\min\{c,1\}$, so that $0<m\le1$. Since
\[
\langle T^*Tx,x\rangle=\|Tx\|^2\ge c^2\|x\|^2\ge m^2\|x\|^2 \qquad(x\in \mathcal{H}),
\]
we have $T^*T\ge m^2I$. An argument similar to the one used for $M$ gives
\begin{equation}\label{lowerbd}
|T^n|^{1/(n+1)}-|T^{n+1}|^{1/(n+1)}
\le \left(m^{-1/(n+1)}-1\right)|T^{n+1}|^{1/(n+1)}
\le M\left(m^{-1/(n+1)}-1\right)I.
\end{equation}
Combining inequalities~\eqref{upperbd} and~\eqref{lowerbd}, we obtain
\begin{equation}\label{maxbd}
\bigl\| |T^{n+1}|^{1/(n+1)}-|T^n|^{1/(n+1)} \bigr\|
\le M\max\Bigl\{M^{1/(n+1)}-1,m^{-1/(n+1)}-1\Bigr\}.
\end{equation}
The right-hand side of~\eqref{maxbd} tends to $0$ as $n\to\infty$, which completes the proof.
\end{proof}

Lemmas~\ref{convparcomoper} and~\ref{nexpowbb} immediately yield the following result.

\begin{corollary}\label{bblcon}
If $T$ is bounded below, then
\begin{equation}\label{normconv}
\lim_{n\to\infty}\bigl\||T^{n+1}|^{1/(n+1)}-|T^n|^{1/n}\bigr\|=0.
\end{equation}
\end{corollary}

Clearly, \eqref{normconv} implies
\begin{equation}\label{wotconv}
\operatorname*{WOT-lim}_{n\to\infty}\left(|T^{n+1}|^{1/(n+1)}-|T^n|^{1/n}\right)=0.
\end{equation}

\begin{proposition}\label{boubelcon}
If \eqref{wotconv} holds, then $\mathfrak{C}$ is connected in the $\operatorname{WOT}$.
\end{proposition}
\begin{proof}
By Lemma~\ref{wotball} and Remark~\ref{welldef}, we may regard $\mathfrak C$ as the cluster set of $\{|T^n|^{1/n}\}_{n=1}^{\infty}$ formed in $(\overline{\operatorname{ball}}_{2\|T\|},\operatorname{WOT})$, which is compact and metrized by the metric $d$ defined in \eqref{wotmetric}. The definition of $d$ implies
\[
d\bigl(|T^{n+1}|^{1/(n+1)},|T^n|^{1/n}\bigr)
=d\bigl(|T^{n+1}|^{1/(n+1)}-|T^n|^{1/n},0\bigr).
\]
Condition~\eqref{wotconv} implies that this distance tends to $0$.
Hence, Lemma~\ref{barone} shows that $(\mathfrak{C},\operatorname{WOT})$ is connected.
\end{proof}

Thus, we have the following chain of implications:
\[
T\ \text{bounded below} \ \Longrightarrow\ \eqref{normconv} \ \Longrightarrow\ \eqref{wotconv} \ \Longrightarrow\ \mathfrak C\ \text{is connected in the WOT}.
\]
It is therefore natural to ask whether $\mathfrak C$ is connected in the $\operatorname{WOT}$ for every operator. The answer is no: Example~\ref{wsnbb} provides a counterexample.

\section{Asymptotic Numerical Radius and Hyperinvariant Subspaces}\label{hyperinvsub}

For $x\in \mathcal{H}$, set
\[
w_a(T,x) \coloneqq\max W_a(T,x).
\]
The maximum exists, as Theorem~\ref{Qxcompactinterval} shows. We have
\begin{equation}\label{waequidef}
w_a(T,x) =\max_{A\in\mathfrak C}\langle Ax,x\rangle =\limsup_{n\to\infty}\bigl\langle |T^n|^{1/n}x,x\bigr\rangle.
\end{equation}
Here the first equality follows from Proposition~\ref{WaC}, while the second follows directly from the definitions of $W_a(T,x)$ and the limit superior.

Recall that the asymptotic numerical radius of $T$ is given by
\[
w_a(T) =\sup W_a(T).
\]
Equivalently,
\begin{align}\label{waequiwhy}
w_a(T) =\sup_{\|x\|=1}w_a(T,x) =\sup_{\|x\|=1}\sup_{A\in\mathfrak C}\langle Ax,x\rangle =\sup_{A\in\mathfrak C}\sup_{\|x\|=1}\langle Ax,x\rangle =\sup_{A\in\mathfrak C}\|A\|.
\end{align}
Indeed, the first equality follows from the definitions, the second from \eqref{waequidef}, and the fourth from the positivity of every $A\in\mathfrak C$.

We recall the following standard inequality, which will be used later.

\begin{lemma}[H\"older--McCarthy inequality {\cite[Lemma~2.1]{mccarthy1967Cp}}]\label{hmineq}
If $A\ge0$, $0<\alpha\le1$ and $x\in \mathcal{H}$ is a unit vector, then
\begin{equation}\label{homc}
\langle A^{\alpha}x,x\rangle\le\langle Ax,x\rangle^{\alpha}.
\end{equation}
\end{lemma}

It is classical (see, for example, \cite[Proposition~1.5.1(d)]{wu2021Numerical}) that
\begin{equation}\label{comasynorm}
r(T)\le w(T)\le\|T\|.
\end{equation}
The next proposition complements these inequalities with a comparison for the asymptotic numerical radius, together with its local version.

\begin{proposition}\label{impoineq}
For every unit vector $x\in\mathcal{H}$,
\begin{equation}\label{waxlrx}
w_a(T,x)\le r(T,x).
\end{equation}
Consequently,
\begin{equation}\label{walr}
w_a(T)\le r(T).
\end{equation}
\end{proposition}

\begin{proof}
Since $\||T^n|x\|=\|T^nx\|$, the H\"older--McCarthy inequality \eqref{homc} and the Cauchy--Schwarz inequality give
\begin{equation}\label{ineqchain}
\langle |T^n|^{1/n}x,x\rangle\le \langle |T^n|x,x\rangle^{1/n}\le \|T^nx\|^{1/n}\qquad(n\ge1).
\end{equation}
Taking the limit superior in \eqref{ineqchain} yields \eqref{waxlrx}. It follows from \eqref{waequiwhy} that
\[
w_a(T) =\sup_{\|x\|=1}w_a(T,x)\le \sup_{\|x\|=1} r(T,x)\le r(T).
\]
\end{proof}

We next recall the classical construction of hyperinvariant subspaces from local spectral theory and record its subsequential version.

Let $\Gamma\coloneqq\{n_j\}_{j=1}^{\infty}$ be a strictly increasing sequence of positive integers. For $x\in\mathcal{H}$, define
\[
r^\Gamma(T,x)\coloneqq\limsup_{j\to\infty}\|T^{n_j}x\|^{1/n_j}.
\]
As in the case of the local spectral radius, $r^\Gamma(T,\cdot)$ has the following properties for all $x,y\in\mathcal{H}$:
\begin{enumerate}
\item $r^\Gamma(T,\alpha x)=r^\Gamma(T,x)$ for every $\alpha\in\mathbb C\setminus\{0\}$;
\item $r^\Gamma(T,x+y)\le\max\{r^\Gamma(T,x),r^\Gamma(T,y)\}$;
\item $r^\Gamma(T,Sx)\le r^\Gamma(T,x)$ for every $S\in\mathcal{B}(\mathcal{H})$ commuting with $T$.
\end{enumerate}
The above shows that, for every $\tau\ge0$,
\[
M_\tau^\Gamma(T)\coloneqq\{x\in\mathcal{H}:r^\Gamma(T,x)\le\tau\}
\]
is a hyperinvariant linear manifold for $T$. Consequently, $\overline{M_\tau^\Gamma(T)}$ is a hyperinvariant subspace of $T$. When $\Gamma=\{n\}_{n=1}^{\infty}$, $r^\Gamma(T,x)$ and $M_\tau^\Gamma(T)$ are precisely $r(T,x)$ and $M_\tau(T)$, respectively.

\begin{lemma}\label{ynotin}
Let $A\in\mathfrak C$, and let $\Gamma\coloneqq\{n_j\}_{j=1}^{\infty}$ be a strictly increasing sequence such that
\[
|T^{n_j}|^{1/n_j}\xrightarrow{\operatorname{WOT}}A.
\]
If $\tau\ge0$ and $y_0\in\mathcal{H}$ is a unit vector satisfying
\[
\langle Ay_0,y_0\rangle>\tau,
\]
then
\[
y_0\notin\overline{M_\tau^\Gamma(T)}.
\]
\end{lemma}

\begin{proof}
Let $u\in M_\tau^\Gamma(T)$ be a unit vector. By \eqref{ineqchain} and the WOT convergence,
\[
\langle Au,u\rangle=\lim_{j\to\infty}\langle |T^{n_j}|^{1/n_j}u,u\rangle\le\limsup_{j\to\infty}\|T^{n_j}u\|^{1/n_j}=r^\Gamma(T,u)\le\tau.
\]
By continuity, the same inequality holds for every unit vector in $\overline{M_\tau^\Gamma(T)}$. Since $\langle Ay_0,y_0\rangle>\tau$, the conclusion follows.
\end{proof}

The next two results give two criteria ensuring that $\overline{M_\tau^\Gamma(T)}$ is nontrivial.

\begin{proposition}\label{master}
Suppose that there exist a nonzero vector $x_0\in\mathcal{H}$ and a unit vector $y_0\in\mathcal{H}$ such that
\[
\limsup_{n\to\infty} \left( \langle |T^n|^{1/n}y_0,y_0\rangle-\|T^nx_0\|^{1/n} \right)>0.
\]
Then $T$ has a nontrivial hyperinvariant subspace that contains $x_0$ but does not contain $y_0$.
\end{proposition}

\begin{proof}
Set
\[
A_n\coloneqq|T^n|^{1/n},\qquad a_n\coloneqq\|T^nx_0\|^{1/n},\qquad b_n\coloneqq\langle A_ny_0,y_0\rangle.
\]
Choose a strictly increasing sequence $\Gamma\coloneqq\{n_j\}_{j=1}^{\infty}$ such that
\[
b_{n_j}-a_{n_j}\longrightarrow\limsup_{n\to\infty}\left(b_n-a_n\right)>0.
\]
The sequences $\{\|A_{n_j}\|\}_{j=1}^{\infty}$, $\{a_{n_j}\}_{j=1}^{\infty}$ and $\{b_{n_j}\}_{j=1}^{\infty}$ are bounded. Hence, by Lemma~\ref{wotball} and the Bolzano--Weierstrass theorem, after passing to a further subsequence and replacing $\Gamma$ by it, we may assume that, for some $A\in\mathfrak C$ and some $\alpha,\beta\ge0$,
\[
A_{n_j}\xrightarrow{\operatorname{WOT}}A,\qquad a_{n_j}\to\alpha,\qquad b_{n_j}\to\beta.
\]
It follows that
\[
r^\Gamma(T,x_0)
=\lim_{j\to\infty}\|T^{n_j}x_0\|^{1/n_j}
=\alpha
<\beta
=\lim_{j\to\infty}\langle A_{n_j}y_0,y_0\rangle
=\langle Ay_0,y_0\rangle.
\]
Lemma~\ref{ynotin}, with $\tau=\alpha$, gives $y_0\notin\overline{M_\alpha^\Gamma(T)}$, whereas $0\ne x_0\in M_\alpha^\Gamma(T)$. Thus $\overline{M_\alpha^\Gamma(T)}$ is the desired nontrivial hyperinvariant subspace.
\end{proof}

\begin{theorem}\label{nontriinv}
If there exists a nonzero vector $x_0\in \mathcal{H}$ such that
\[
r(T,x_0)<w_a(T),
\]
then $\overline{M_{r(T,x_0)}(T)}$ is a nontrivial hyperinvariant subspace of $T$ containing $x_0$, and the closed cyclic subspace $\overline{\operatorname{span}\{T^n x_0:n\ge0\}}$ generated by $x_0$ under $T$ is a nontrivial invariant subspace of $T$.
\end{theorem}

\begin{proof}
Set $\tau\coloneqq r(T,x_0)$. By \eqref{waequiwhy}, there exists a unit vector $y_0\in\mathcal{H}$ such that $w_a(T,y_0)>\tau$. Hence, by \eqref{waequidef},
\[
\limsup_{n\to\infty}\langle |T^n|^{1/n}y_0,y_0\rangle=w_a(T,y_0)>\tau.
\]
Choose a strictly increasing sequence $\Gamma\coloneqq\{n_j\}_{j=1}^{\infty}$ such that
\[
\langle |T^{n_j}|^{1/n_j}y_0,y_0\rangle\to w_a(T,y_0).
\]
By Lemma~\ref{wotball}, after passing to a further subsequence and replacing $\Gamma$ by it, we may assume that
\[
|T^{n_j}|^{1/n_j}\xrightarrow{\operatorname{WOT}}A
\]
for some $A\in\mathfrak C$. Then $\langle Ay_0,y_0\rangle=w_a(T,y_0)>\tau$, so Lemma~\ref{ynotin} gives
\[
y_0\notin\overline{M_\tau^\Gamma(T)}.
\]
Since $r^\Gamma(T,x)\le r(T,x)$ for every $x\in\mathcal{H}$,
\[
M_\tau(T)\subseteq M_\tau^\Gamma(T).
\]
Hence $y_0\notin\overline{M_\tau(T)}$, so $\overline{M_\tau(T)}\ne\mathcal{H}$. On the other hand, $0\ne x_0\in M_\tau(T)$. Thus $\overline{M_\tau(T)}$ is a nontrivial hyperinvariant subspace of $T$ containing $x_0$.

Finally, $\overline{M_\tau(T)}$ is closed, is invariant under $T$, and contains $x_0$; hence it contains
\[
\overline{\operatorname{span}\{T^n x_0:n\ge0\}},
\]
which is therefore a nontrivial invariant subspace of $T$.
\end{proof}

Recall from the Introduction that Bourdon's theorem states that if $T$ is hyponormal and there exists a nonzero vector $x_0\in\mathcal{H}$ such that
\[
r(T,x_0)<\|T\|,
\]
then $T$ has a nontrivial hyperinvariant subspace. On the other hand, \eqref{comasynorm} and \eqref{walr} imply
\[
w_a(T)\le r(T)\le w(T)\le \|T\|.
\]
We will prove below that the above inequalities become equalities when $T$ is hyponormal, thereby recovering Bourdon's theorem as a special case of Theorem~\ref{nontriinv}.

The key ingredient is the following monotonicity criterion.

\begin{proposition}\label{hyponequi}
Suppose that $\{|T^n|^{1/n}\}_{n=1}^{\infty}$ has a subsequence
\[
\{|T^{n_j}|^{1/n_j}\}_{j=1}^{\infty}
\]
which is an increasing operator sequence. Then
\[
w_a(T)=r(T).
\]
If, in addition, $n_1=1$, then
\[
w_a(T)=r(T)=w(T)=\|T\|.
\]
\end{proposition}

\begin{proof}
The increasing sequence $\{|T^{n_j}|^{1/n_j}\}_{j=1}^{\infty}$ is bounded above by $\|T\|I$ and hence, by \cite[Lemma~I.6.4]{davidson1996CAlgebras}, converges in the $\operatorname{SOT}$ to some positive operator $A$. Consequently,
\[
|T^{n_j}|^{1/n_j}\to A\qquad(\operatorname{WOT}),
\]
so $A\in\mathfrak C$. Moreover, $A\ge |T^{n_j}|^{1/n_j}$ for every $j\ge1$, and hence $\|A\|\ge\|T^{n_j}\|^{1/n_j}$. Therefore, by \eqref{waequiwhy} and the spectral radius formula,
\[
w_a(T)\ge\|A\|\ge\lim_{j\to\infty}\|T^{n_j}\|^{1/n_j}=r(T).
\]
The reverse inequality follows from \eqref{walr}.

If $n_1=1$, then $A\ge |T|$. Combining the above with \eqref{comasynorm}, we obtain
\[
r(T)=w_a(T)\ge\|A\|\ge\|T\|\ge w(T) \ge r(T).
\]
\end{proof}

We now identify several generalizations of hyponormal operators that satisfy the monotonicity hypothesis of Proposition~\ref{hyponequi}. Following Fujii and Nakatsu \cite{fujii1975Subclasses} and Aluthge \cite{aluthge1990Hyponormal}, for any $p>0$, an operator $S\in\mathcal{B}(\mathcal{H})$ is said to be \emph{$p$-hyponormal} if
\[
(S^*S)^p\ge(SS^*)^p,
\]
where the case $p=1$ is precisely hyponormality; when $p>1$, the L\"owner--Heinz inequality \eqref{lhe} reduces the situation to $p=1$, so it suffices for our purposes to consider $0<p\le1$. Tanahashi \cite{tanahashi1999Loghyponormal} introduced log-hyponormal operators: an invertible operator $S$ is said to be \emph{log-hyponormal} if
\[
\log(S^*S)\ge\log(SS^*).
\]

Furuta, Ito, and Yamazaki~\cite{furuta1998Subclass} introduced \emph{class $\mathbf{A}$}: an operator $S$ belongs to class $\mathbf{A}$ if
\begin{equation}\label{classA}
|S^2|\ge |S|^2.
\end{equation}
They also proved that every log-hyponormal operator belongs to class $\mathbf{A}$. Yamazaki~\cite[Theorems~1 and~2]{yamazaki1999Extensions} subsequently obtained the stronger result that every $p$-hyponormal or log-hyponormal operator satisfies
\begin{equation}\label{wholechain}
|S|^2\le |S^2|\le\cdots\le |S^n|^{2/n}\le |S^{n+1}|^{2/(n+1)}\le\cdots.
\end{equation}
In particular, every $p$-hyponormal operator also belongs to class $\mathbf{A}$. It was later shown that the defining inequality \eqref{classA} already implies the entire chain \eqref{wholechain}; see Ito~\cite[Theorem~1]{ito2002Classes}. This is precisely what enables us to apply Proposition~\ref{hyponequi} to any operator belonging to class $\mathbf{A}$.

\begin{corollary}\label{phyponormal}
Let $k_0$ be a positive integer and set $S\coloneqq T^{k_0}$. If $S$ belongs to class $\mathbf{A}$, then
\[
w_a(T)=r(T).
\]
Moreover, if $k_0=1$, then
\[
w_a(T)=r(T)=w(T)=\|T\|
\]
and $\mathfrak C$ is a singleton.
\end{corollary}

\begin{proof}
By the preceding discussion, the sequence $\{|S^n|^{2/n}\}_{n=1}^{\infty}$ is increasing. Hence, by the L\"owner--Heinz inequality \eqref{lhe} with $\alpha=1/(2k_0)$, the sequence
\[
\left\{|T^{nk_0}|^{1/(nk_0)}\right\}_{n=1}^{\infty}
\]
is increasing. Proposition~\ref{hyponequi} gives $w_a(T)=r(T)$.

If $k_0=1$, the additional conclusion of Proposition~\ref{hyponequi} gives the remaining equalities, while its proof shows that $\{|T^n|^{1/n}\}_{n=1}^{\infty}$ converges in the $\operatorname{WOT}$. Hence, $\mathfrak C$ is a singleton.
\end{proof}

\section{Examples}\label{exam}

We recall the basic computation for (unilateral) weighted shifts. Let $\{\xi_k\}_{k=1}^{\infty}$ be an orthonormal basis for $\mathcal{H}$, and let $\{w_k\}_{k=1}^{\infty}$ be a bounded scalar sequence. Define the weighted shift $T$ on $\mathcal{H}$ with weight sequence $\{w_k\}_{k=1}^{\infty}$ by
\[
T\xi_k=w_k\xi_{k+1}\qquad(k\ge1).
\]
By \cite[Corollary~1]{shields1974Weighted}, every weighted shift with weight sequence $\{w_k\}_{k=1}^{\infty}$ is unitarily equivalent to the weighted shift with weight sequence $\{|w_k|\}_{k=1}^{\infty}$. Thus, for the weighted shifts considered below, we assume that $w_k\ge0$ for every $k\ge1$.

For $A_n\coloneqq |T^n|^{1/n}$, a direct computation shows that $A_n$ is diagonal with respect to $\{\xi_k\}_{k=1}^{\infty}$,
\[
A_n=\operatorname{diag}(a_1^{(n)},a_2^{(n)},\ldots),
\]
where
\[
a_k^{(n)}\coloneqq(w_k w_{k+1}\cdots w_{k+n-1})^{1/n}\qquad(k,n\ge1).
\]
We retain the above notation throughout the remainder of this section whenever weighted shifts are considered.

The following proposition completely describes $\mathfrak C$ for bounded below weighted shifts.

\begin{proposition}\label{weishi}
Let $T$ be a bounded below weighted shift with weight sequence $\{w_k\}_{k=1}^{\infty}$, and set
\[
J\coloneqq[\liminf_{n\to\infty}a_1^{(n)},\,\limsup_{n\to\infty}a_1^{(n)}].
\]
Then $\mathfrak C=\{sI:s\in J\}$, and $W_a(T,x)=W_a(T)=J$ for every unit vector $x\in \mathcal{H}$.
\end{proposition}

\begin{proof}
Since $T$ is bounded below and bounded, there exist constants $m,M$ with $0<m\le w_k\le M$ for all $k\ge1$. For fixed $k$ and $n>k$,
\[
\frac{a_k^{(n)}}{a_1^{(n)}} = \left( \frac{w_{n+1}\cdots w_{n+k-1}} {w_1\cdots w_{k-1}} \right)^{1/n}.
\]
Hence
\[
\left(\frac{m}{M}\right)^{(k-1)/n} \le \frac{a_k^{(n)}}{a_1^{(n)}} \le \left(\frac{M}{m}\right)^{(k-1)/n},
\]
so $a_k^{(n)}/a_1^{(n)}\to1$. Consequently, for every $k\ge1$, the sequences $\{a_k^{(n)}\}_{n=1}^{\infty}$ and $\{a_1^{(n)}\}_{n=1}^{\infty}$ have the same asymptotic behavior: every subsequence of one converges if and only if the corresponding subsequence of the other converges, and in that case the two limits coincide.

Since $T$ is bounded below, Corollary~\ref{bblcon} gives
\[
|a_1^{(n+1)}-a_1^{(n)}|
\le \|A_{n+1}-A_n\|\to0.
\]
Hence, Lemma~\ref{onesidclupoi} gives
\[
\operatorname{Clust}\{a_1^{(n)}\}_{n=1}^{\infty}=J.
\]

We prove the inclusion $\{sI:s\in J\}\subseteq\mathfrak{C}$. Let $s\in J$, and choose a strictly increasing sequence of positive integers $\{n_j\}_{j=1}^{\infty}$ such that $a_1^{(n_j)}\to s$. Then, for every $k\ge1$, \[\|(A_{n_j}-sI)\xi_k\|=|a_k^{(n_j)}-s|\to 0.\]
Since $\{A_{n_j}-sI\}_{j=1}^{\infty}$ is uniformly bounded, and since a bounded net of operators converges to $0$ in the SOT if and only if it converges to $0$ at every vector of a fixed orthonormal basis \cite[Exercise~8.5(b)]{conway2000Course}, it follows that $A_{n_j}\to sI$ in the SOT, and hence also in the WOT. Thus $sI\in\mathfrak{C}$.

We prove the converse inclusion. Let $A\in\mathfrak{C}$, and choose a subsequence $\{A_{n_j}\}_{j=1}^{\infty}$ that converges to $A$ in the WOT. Set \[s\coloneqq \langle A\xi_1,\xi_1\rangle= \lim_{j\to\infty}\langle A_{n_j}\xi_1,\xi_1\rangle=\lim_{j\to\infty} a_1^{(n_j)}\in J.\] Therefore $\lim_{j\to\infty}a_k^{(n_j)}=s$ for every $k\ge1$. Hence, for all $k,\ell\ge1$,
\[
\langle A\xi_k,\xi_\ell\rangle = \lim_{j\to\infty}a_k^{(n_j)}\delta_{k\ell} = s\delta_{k\ell} = \langle s \xi_k,\xi_\ell\rangle.
\]
Thus $A=sI$, and therefore $\mathfrak C=\{sI:s\in J\}$.

Finally, Proposition~\ref{WaC} gives $W_a(T,x)=W_a(T)=J$ for every unit vector $x\in \mathcal{H}$.
\end{proof}

The interval $J$ need not be degenerate. The following choice of weights, due to Voiculescu \cite[Example~8.4]{haagerup2009Invariant}, yields an example for which $\{|T^n|^{1/n}\}_{n=1}^{\infty}$ does not converge in the WOT.

\begin{example}\label{voiweishi}
Consider the weighted shift with weights
\[
w_n\coloneqq
\begin{cases}
1,&2^j\le n<2^{j+1},\quad j\ \text{even},\\
2,&2^j\le n<2^{j+1},\quad j\ \text{odd}.
\end{cases}
\]

A direct computation gives
\[
\liminf_{n\to\infty}a_1^{(n)}=2^{1/3}, \qquad \limsup_{n\to\infty}a_1^{(n)}=2^{2/3}.
\]
Hence $J=[2^{1/3},2^{2/3}]$. Therefore, by Proposition~\ref{weishi},
\[
\mathfrak C = \{sI:2^{1/3}\le s\le2^{2/3}\},
\]
and
\[
W_a(T,x)=W_a(T)=[2^{1/3},2^{2/3}]
\]
for every unit vector $x\in \mathcal{H}$.
\end{example}

The following example provides a concrete application of Proposition~\ref{master} and, as promised, exhibits an operator $T$ for which $\mathfrak C$ is disconnected in the $\operatorname{WOT}$.

\begin{example}\label{wsnbb}
Choose a strictly increasing sequence of positive integers $\{p_i\}_{i=1}^{\infty}$ with $p_1\ge 2$ such that
\[
\lim_{i\to\infty}\frac{(p_1-1)+\cdots+(p_{i-1}-1)}{p_i-1}=0.
\]
For instance, $p_i=2^{2^i}+1$; indeed,
\[
\frac{2^{2}+2^{2^2}+\cdots+2^{2^{i-1}}}{2^{2^i}} \le \frac{(i-1)2^{2^{i-1}}}{2^{2^i}} = \frac{i-1}{2^{2^{i-1}}} \longrightarrow0.
\]
Let $T$ be the weighted shift with weights
\[
w_k\coloneqq
\begin{cases}
e^{-(p_i-1)}, & k=p_i\ \text{for some }i,\\
1, & \text{otherwise},
\end{cases}
\]
so that $\|T\|=1$ and $T$ is injective but not bounded below.

Let $x\coloneqq\xi_1+\xi_2$ and $y\coloneqq\xi_1-\xi_2$, and set
\[
d_n\coloneqq\langle A_nx,y\rangle=a_1^{(n)}-a_2^{(n)}.
\]
Since $w_1=1$,
\[
a_2^{(n)}=(w_2\cdots w_{n+1})^{1/n}=a_1^{(n)}\,w_{n+1}^{1/n}.
\]
Hence
\[
d_n=a_1^{(n)}\bigl(1-w_{n+1}^{1/n}\bigr)=
\begin{cases}
0, & n+1\notin\{p_i:i\ge1\},\\
(1-e^{-1})a_1^{(p_i-1)}, & n+1=p_i,
\end{cases}
\]
while
\[
a_1^{(p_i-1)} =\Bigl(\,\prod_{j<i}e^{-(p_j-1)}\Bigr)^{1/(p_i-1)} =\exp\Bigl(-\frac{\sum_{j<i}(p_j-1)}{p_i-1}\Bigr)\to1.
\]
Therefore $d_{p_i-1}=\bigl(1-e^{-1}\bigr)a_1^{(p_i-1)}\to1-e^{-1}$. Every term of $\{d_n\}_{n=1}^{\infty}$ is either $0$ or a term of $\{d_{p_i-1}\}_{i=1}^{\infty}$, both index sets are infinite, and $d_{p_i-1}\to1-e^{-1}\neq0$. Therefore
\begin{equation}\label{dnclust}
\operatorname{Clust}\{\langle A_nx,y\rangle\}_{n=1}^{\infty}
=\operatorname{Clust}\{d_n\}_{n=1}^{\infty}=\{0,\;1-e^{-1}\},
\end{equation}
which is disconnected. The same argument as in the proof of Proposition~\ref{WaC}, with $\langle Ax,x\rangle$ replaced by $\langle Ax,y\rangle$, gives
\[
\operatorname{Clust}\{\langle A_nx,y\rangle\}_{n=1}^{\infty}=\{\langle Ax,y\rangle:A\in\mathfrak C\}.
\]
Thus the map
\[
\mathfrak C\longrightarrow\{0,\;1-e^{-1}\},\qquad A\longmapsto\langle Ax,y\rangle,
\]
is a $\operatorname{WOT}$-continuous surjection. Hence, $\mathfrak C$ is disconnected in the $\operatorname{WOT}$.

This example also gives a concrete application of Proposition~\ref{master}. Since
\[
\langle |T^n|^{1/n}\xi_1,\xi_1\rangle-\|T^n\xi_2\|^{1/n}
=a_1^{(n)}-\||T^n|\xi_2\|^{1/n}
=a_1^{(n)}-a_2^{(n)}=d_n,
\]
it follows from \eqref{dnclust} that
\[
\limsup_{n\to\infty}\left(\langle |T^n|^{1/n}\xi_1,\xi_1\rangle-\|T^n\xi_2\|^{1/n}\right)
=\limsup_{n\to\infty}d_n=1-e^{-1}>0.
\]
Hence Proposition~\ref{master} yields a nontrivial hyperinvariant subspace containing $\xi_2$ but not $\xi_1$.

\end{example}

We now consider the sum of a multiplication operator and the Volterra operator and give a concrete application of Theorem~\ref{nontriinv}.
\begin{example}\label{mpv}
Let $\mathcal H=L^2([0,1])$. Consider 
\[
(Tf)(x)\coloneqq xf(x)+\int_0^x f(t)\,dt= ((M_x+V)f)(x) 
\]
for $f\in\mathcal H$ and almost every $x\in[0,1]$, where $M_x$ and $V$ are the multiplication operator by the independent variable and the Volterra operator, respectively.

It is known that (see \cite[pp.~38--40]{sarason1974Invariant} and \cite[p.~397, Observation~4.2]{ong1981Invariant})
\begin{equation}\label{mulvoliden}
VT^n=M_x^nV,\qquad
T^n=M_x^n+nM_x^{n-1}V,\qquad
T^*1=1,
\end{equation}
where $1$ denotes the constant function $1$ on $[0,1]$.
Let $P_1$ denote the orthogonal projection onto $\mathbb C1$. Since \eqref{mulvoliden} yields
\[
\langle T^{*n}T^nf,f\rangle
=\|T^nf\|^2
\ge |\langle T^nf,1\rangle|^2
=|\langle f,1\rangle|^2
=\|\langle f,1\rangle1\|^2
=\|P_1f\|^2,
\]
we have $T^{*n}T^n=|T^n|^2\ge P_1$. Therefore, by the L\"owner--Heinz inequality \eqref{lhe},
\[
|T^n|^{1/n}\ge P_1.
\]
Using \eqref{mulvoliden}, we obtain
\[
1=\langle P_1 1,1\rangle
\le \langle |T^n|^{1/n}1,1\rangle
\le \bigl\||T^n|^{1/n}\bigr\|
\le \bigl\|M_x^n+nM_x^{n-1}V\bigr\|^{1/n}
\le (1+n\|V\|)^{1/n}
\longrightarrow 1.
\]
By the definition of $w_a(T,\cdot)$ and the spectral radius formula, $w_a(T,1)=r(T)=1$.
By \eqref{waequiwhy} and \eqref{walr}, $w_a(T)=1$. We next compute $r(T,f)$. For $0\ne f\in\mathcal H$, put $F\coloneqq Vf$. Note that $F$ is absolutely continuous and $F'=f$ almost everywhere. Since $f\neq 0$, we have $F\neq 0$. Therefore
\[
\alpha_f\coloneqq\sup\{x\in[0,1]:F(x)\neq 0\}>0.
\]
Both $f$ and $F$ vanish almost everywhere on $(\alpha_f,1]$. Hence, by \eqref{mulvoliden},
\[
\|T^nf\|
\le \|M_x^nf\|+n\|M_x^{n-1}F\|
\le \alpha_f^n\|f\|+n\alpha_f^{n-1}\|F\|,
\]
and therefore
\[
\limsup_{n\to\infty}\|T^nf\|^{1/n}\le \alpha_f.
\]
On the other hand, the intertwining relation in \eqref{mulvoliden} gives
\[
\|V\|\,\|T^nf\|\ge \|VT^nf\|=\|M_x^nF\|.
\]
Hence
\[
\liminf_{n\to\infty}\|T^nf\|^{1/n}\ge \liminf_{n\to\infty}\|M_x^nF\|^{1/n}.
\]

For any $a\in (0,\alpha_f)$, we have
\[
a \|F\chi_{[a,\alpha_f]}\|^{1/n} \le\|M_x^nF\|^{1/n}\le\alpha_f \|F\|^{1/n},
\]
where $\chi_{[a,\alpha_f]}$ denotes the characteristic function of $[a,\alpha_f]$. Since both $F\chi_{[a,\alpha_f]}$ and $F$ are nonzero vectors, taking $n\to\infty$ and then $a\uparrow\alpha_f$, we obtain
\[
\lim_{n\to\infty}\|M_x^nF\|^{1/n}=\alpha_f.
\]
Consequently,
\[
r(T,f)=\lim_{n\to\infty}\|T^nf\|^{1/n}=\alpha_f=\sup\{x\in[0,1]:(Vf)(x)\ne0\}.
\]
Recalling the hyperinvariant linear manifold $M_\tau(T)=\{f\in\mathcal H:r(T,f)\le\tau\}$, we obtain, for every $\tau\in[0,1]$,
\[
M_\tau(T)=\{f\in\mathcal H:Vf=0\text{ a.e. on }(\tau,1]\}=\ker\bigl(M_{\chi_{(\tau,1]}}V\bigr).
\]
Here $M_{\chi_{(\tau,1]}}$ denotes multiplication by the characteristic function of $(\tau,1]$. The closedness of $M_\tau(T)$ follows from the kernel representation above, so $M_\tau(T)$ is a hyperinvariant subspace of $T$ for every $\tau\in[0,1]$. This family of hyperinvariant subspaces is a special case of Sarason's characterization of the invariant subspace lattice of $T$; see \cite[Section~9]{sarason1974Invariant}.

For any $\tau\in(0,1)$, define a nonzero vector
\[
f_\tau(x)\coloneqq
\begin{cases}
1, & 0\le x<\tau/2,\\
-1, & \tau/2\le x<\tau,\\
0, & \tau\le x\le1.
\end{cases}
\]

Hence
\[
r(T,f_\tau)=\tau<1=w_a(T,1).
\]
Thus, by Theorem~\ref{nontriinv}, $M_\tau(T)$ is a nontrivial hyperinvariant subspace of $T$ for every $\tau\in(0,1)$.

\end{example}


\section*{Declaration of generative AI and AI-assisted technologies in the manuscript preparation process}

The authors developed the concepts, formulated and proved all theorems, propositions, and lemmas presented in this work, and determined the structure and organization of the manuscript. OpenAI's ChatGPT (GPT-5.6 Sol) and Anthropic's Claude (Fable 5) were used to assist with the construction of Examples~\ref{wsnbb} and~\ref{mpv} and the associated computations. These tools were also used throughout the manuscript for language polishing and literature searches. The authors independently verified all AI-assisted constructions and computations and checked all references suggested by these tools against the original sources. The authors take full responsibility for the final manuscript.


\par\bigskip
\end{document}